\documentclass[10pt,a4paper]{article}
\usepackage{geometry}
\usepackage[english]{babel}
\usepackage{amsmath, bm}
\usepackage{graphicx}
\usepackage{ marvosym }
\usepackage[utf8]{inputenc}
\usepackage{mathtools}   
\usepackage{geometry}
\usepackage{amsfonts}
\usepackage{amssymb}
\usepackage{amsthm}
\usepackage{t1enc}
\usepackage[titles]{tocloft}
\usepackage{makeidx}
\usepackage{wasysym}
\usepackage{stmaryrd}
\usepackage{calc}  
\usepackage{enumitem} 
\usepackage[refpage]{nomencl}

\usepackage[colorlinks=true,urlcolor=blue, linkcolor=blue,pageanchor=false, backref ]{hyperref}
\usepackage{algpseudocode}
\usepackage{tikz}
\usetikzlibrary{arrows}
\usepackage{indentfirst}
\usepackage{xcolor}
\usepackage{ dsfont }
\usepackage{float}
\usepackage[abbrev,msc-links,backrefs]{amsrefs} 
\usepackage{doi}

\renewcommand{\PrintDOI}[1]{\doi{#1}}

\theoremstyle{plain}
\newtheorem{thm}{Theorem}[section]

\newtheorem{claim}[thm]{Claim}

\newtheorem{prop}[thm]{Proposition}

\newtheorem*{prop*}{Proposition}
\newtheorem*{seged*}{Sublemma}

\newtheorem{lem}[thm]{Lemma}

\newtheorem*{cond*}{Sondition}

\newtheorem*{lem*}{Lemma}
\theoremstyle{definition}

\newtheorem*{defn*}{Definition}

\newtheorem{fel*}[thm]{Exercise}

\newtheorem*{megf*}{Observation}
\theoremstyle{remark}
\newtheorem{rem}[thm]{Remark}

\newtheorem*{rem*}{Remark}

\newenvironment{sbiz}{\par\noindent{\itshape Proof:}\ }{\rule{1ex}{1ex}}

\title{The Complete Lattice of Erd\H os-Menger Separations}
\author{Attila Joó \thanks{Affiliation: University of Hamburg and Alfréd Rényi Institute of Mathematics. Funding was provided 
by the Alexander von Humboldt Foundation.
 Email: {\tt attila.joo@uni-hamburg.de
 } }}
\date{2019}
\begin{document}
\maketitle
\begin{abstract}
F. Escalante and T. Gallai studied in the seventies the structure of different kind of  separations and cuts
between a vertex pair in a (possibly infinite) graph. One of their results is that if there is a finite separation, then the 
optimal (i.e. minimal sized) separations form a finite distributive lattice with respect to a natural partial order. Furthermore, 
any finite
distributive lattice can be 
represented this way. 

If there is no finite separation then cardinality is a too rough measure to capture  being ``optimal''. 
Menger's theorem provides a structural characterization of optimality if there is a finite separation.  We use this 
characterization to define  Erd\H os-Menger 
separations even if there is no finite separation. The generalization of Menger's theorem to infinite 
graphs (which was not available until 2009)  ensures that  Erd\H os-Menger 
separations always exist.
We show that they  form a 
complete lattice with respect to the partial order given by Escalante and every complete lattice can be represented this way.
\end{abstract}

\section{Introduction}

The investigation of the structure of several type of separations (i.e. vertex cuts) and cuts in graphs  has been started  in the 
seventies by 
F. Escalante and T. Gallai. For their original papers see \cite{escalante1972schnittverbande} and 
\cite{escalante1974note} and for an English survey about these and further results in this area we recommend the chapter 
``Lattices Related to Separation  in  Graphs'' of \cite{sauer1993finite} by R. Halin. 

Among other results, it was 
discovered by Escalante that if there is a finite separation between two vertex sets in a given graph, then the 
optimal 
(minimal sized) separations form a finite distributive lattice with respect to a natural partial order. Furthermore, any finite 
distributive lattice can be 
represented this way. 
Without having a finite separation it was unclear which separations we should consider ``optimal''. By Menger's theorem, a 
finite  separation  $ S $ between two vertex sets
is optimal if and only if there is a system of  disjoint paths  joining them such that $ S $ consists of 
choosing exactly one
vertex from each of these paths. Based on this characterisation, the concept of optimal separation can be 
interpreted without having a  finite separation.  The generalization of Menger's theorem to infinite graphs (see 
\cite{aharoni2009menger}) ensures that this definition makes sense, this kind of  
separation always exists. Since the infinite version of Menger's  theorem was conjectured by P. Erd\H os,  we call them   
Erd\H os-Menger 
separations. Our main 
result is that the Erd\H os-Menger separations always form a complete lattice and every  complete lattice 
can be represented as an Erd\H os-Menger 
separation lattice.  We are working with 
digraphs but our results remain true in undirected graphs as well with obvious modification of the proofs. The paper is 
structured as follows. We introduce few notation in the next section. The main result is discussed in the third section. Finally 
at the Appendix we show by an example that to the contrary of the finite case the Erd\H os-Menger separation lattice is not 
necessarily a sublattice of the 
minimal separation lattice.
\section{Notation}
  Let $ D=(V,E) $ be a possibly infinite 
digraph and 
$ A, B\subseteq V $. Later we will omit $ D $ from our notation whenever it is fixed or clear from the context. A finite 
directed 
path $ P $  is an $ \boldsymbol{A \rightarrow B }$ \textbf{path} if exactly its first vertex is in $ A $ and exactly its last is in $ B $.  
Let $ 
\boldsymbol{\mathfrak{D}_D(A,B)} $ consist of the 
systems $ \mathcal{P} $ 
of (pairwise) disjoint $ A\rightarrow B $ paths. We write $ \boldsymbol{V_{\text{first}}(\mathcal{P}) }$ for the set of the first 
vertices of the paths in $ \mathcal{P} $ and we define $ \boldsymbol{V_{\text{last}}(\mathcal{P})}$ analogously.   Let us 
write 
$ \boldsymbol{\mathfrak{M}_D(A,B)} $ for the set of the minimal $ AB $-separations in $ D $, i.e., those $ S\subseteq V 
$ for
which every $ A\rightarrow B $ path in $ D $ meets $ S $ and $ S $ is $ \subseteq $-minimal with respect to this 
property.  We 
consider the following relation 
$\trianglelefteq $ on $ 
\mathfrak{M}_D(A,B) 
$.  Let $ 
S\trianglelefteq T $ if 
$ S $ separates $ T $ from $ A $ (i.e. $ S $ meets every $ A\rightarrow T $ path of $ D $). It is known that $ 
\mathfrak{M}_D(A,B)  $ is a complete lattice (see Proposition \ref{inf char}), we use \textbf{inf} and \textbf{sup} always 
with 
respect to this lattice.  A vertex set $ S $ is orthogonal to a 
system $ \mathcal{P} $ of disjoint 
paths (we write $ \boldsymbol{S\bot\mathcal{P} } $) if $ S $ consists of choosing exactly one vertex from each path of 
$ \mathcal{P} $. The formal definition of the \textbf{Erd\H os-Menger separations} is the following.

\[ \boldsymbol{\mathfrak{S}_{D}(A,B)}:=\{ S\in \mathfrak{M}_D(A,B): \exists 
\mathcal{P}\in \mathfrak{D}_D(A,B)\text{ 
 with 
}S\bot\mathcal{P} \}. \]

The non-emptiness of  $ \mathfrak{S}_{D}(A,B) $ in the general case is guaranteed by the Aharoni-Berger theorem (see 
\cite{aharoni2009menger}).  Finally let 
\begin{align*}
 &\boldsymbol{\mathfrak{S}^{-}_{D}(A,B)}:=\{S\in \mathfrak{M}_{D}(A,B): \exists \mathcal{P}\in 
 \mathfrak{D}_{D}(A,S)\text{ 
 with }
  V_{\text{last}}(\mathcal{P})=S  \}\\ 
  &\boldsymbol{\mathfrak{S}^{+}_{D}(A,B)}:=\{S\in \mathfrak{M}_{D}(A,B): \exists \mathcal{P}\in 
  \mathfrak{D}_{D}(S,B)\text{ 
  with }
   V_{\text{first}}(\mathcal{P})=S  \}.
 \end{align*} 

\section{Main result}
\subsection{Preliminaries}
We will need some of the basic facts discovered  by Escalante.  He formulated originally these results for undirected 
graphs and for separations between vertex pairs in his paper \cite{escalante1972schnittverbande} (which  is in German). We 
will give here all the necessary details to make the paper
 self-contained . From now on let a 
digraph $ D=(V,E) $ and $ A,B\subseteq V $ be fixed.  If a statement is ``symmetric'', then 
we prove just one half of it without mentioning this every time explicitly.

The role of $ A $ and $ B $ are seemingly not symmetric in the definition of $ \trianglelefteq $ (the definition based on $ A $ 
and does not mention $ B $). The following Proposition restore the symmetry. 
\begin{prop}\label{simm ord}
Let $ S,T\in \mathfrak{M}(A,B) $. Then $ T $ separates $ S $ from $ A $ if and only if $ S $ separates $ B$ from $ T$. 
\end{prop}
\begin{proof}
Assume that $ T $ separates $ S$ from $ A $. Let $ P $ be a $ T\rightarrow B $ path starting  at $ u $. Pick an $ A\rightarrow T $ 
path $ Q $ terminating at $ u $ (it exists by the 
minimality of $ T $). The path $ Q $ cannot meet $ S $ before $ u $ because $ T $ separates $ S $ from $ A$. Let $ R $ be the 
path 
that we obtain by uniting $ Q $ and $ P $. It is an $ A\rightarrow B $ path therefore it meets $ S $. Thus $ P $ meets $ S $.
\end{proof}
\begin{prop}
$ \trianglelefteq $ is a partial order.
\end{prop}
\begin{proof}
The reflexivity and transitivity are obvious. Let $ S,T\in \mathfrak{M}(A,B) $ and assume that $ S\trianglelefteq T $ and $ 
T\trianglelefteq S $. Let $ 
u\in T $ be arbitrary and pick an $ A\rightarrow B $ path $ P $ which meets $ T $ only at $ u $. Then $ S $ cannot have a vertex 
on $ P $ before $ u $ because $ T\trianglelefteq S $. On the other hand, $ S $ cannot have a vertex on $ P $ after $ u $ since $ 
T $ 
separates $ B $ from $ S $ (use Proposition \ref{simm ord} and $ S\trianglelefteq T $). It follows that $ u\in S $ thus $ S 
\supseteq T $ 
 and by 
minimality $ 
S=T $.
\end{proof}

\begin{prop}\label{inf char}
$ (\mathfrak{M}(A,B),\trianglelefteq) $ is a complete lattice, where for a nonempty
$ \mathcal{S}\subseteq \mathfrak{M}(A,B),\  \inf\mathcal{S} $ consists of those  
$s\in \bigcup \mathcal{S} $  which are reachable 
from 
$ A $ without touching any other element of $ \bigcup \mathcal{S} $.
\end{prop}
\begin{proof}
The set we claimed to be  $ \inf \mathcal{S} $, say $ S $,   separates every element of $ 
\mathcal{S} $ from $ A $. Furthermore, if a  $ T\in \mathfrak{M}(A,B) $ separates all the separations in $ \mathcal{S} $ 
from $ A $, then it separates $ S $ 
from $ 
A $ as well.   

It remains to check  the $ \subseteq $-minimality of $ S $. Let $ s\in S  $ be arbitrary. We need to find an $ A\rightarrow 
B $ path that meets $ S $ only at $ 
s $. By the definition of $ S $, there is an $ A\rightarrow s $ path $ P $ which avoids $ \bigcup \mathcal{S} \setminus 
\{ s
\} $. Pick a $ T\in \mathcal{S} $ for which $ s\in T $. Since $ T$ is a minimal separation, there is a $ S\rightarrow B $ path $ 
Q $ starting
at $ s $. The vertices $ V(Q)\setminus \{ s \} $ are not reachable from $ A $ without touching $ T$ thus they are not in $ S $. 
Hence by 
uniting $ P $ and $ Q $ we obtain a desired $ A\rightarrow B $ path.
\end{proof} 

\begin{thm}[Escalante]
If  $ \mathfrak{M}(A,B) $ has a finite element, then $ \mathfrak{S}(A,B) $ is a finite distributive sublattice of $ 
\mathfrak{M}(A,B) $.
\end{thm}
\begin{proof}
 Let a nonempty $ \mathcal{S}\subseteq  \mathfrak{S}(A,B) $ be given. We fix a maximal-sized element $ \mathcal{P} $ of 
 $ \mathfrak{D}(A,B) $. Note that an  $ S\in \mathfrak{M}(A,B) $ is 
 in $ \mathfrak{S}(A,B) $ iff $ S \bot \mathcal{P} $. Every vertex in $ \inf \mathcal{S} $ is 
 coming from an optimal 
separation and 
hence used by $ \mathcal{P} $. Let $ P\in \mathcal{P} $ be arbitrary and let $ s $ be the first vertex of $ P $ which is in 
$ \inf \mathcal{S} $. There is a $ S\in \mathcal{S} $ such that $ s\in S $. Since $ S\bot \mathcal{P} $, all the vertices of $ P 
$ 
after $ s $ are separated from $ A $ by $ S $ and hence cannot be in $ \inf \mathcal{S} $. Therefore $ \inf \mathcal{S}\bot 
\mathcal{P} $ which means $ S\in \mathfrak{S}(A,B) $. Finally let $ H $ be the digraph consists of $ A,B $ and the paths in 
$ \mathcal{P} $. Then $ 
\mathfrak{S}(A,B)  $ is 
a sublattice of the 
finite, 
distributive lattice $ 
\mathfrak{M}_{H}(A,B) $, 
thus it is distributive.
\end{proof}

\subsection{The complete lattice of the Erd\H os-Menger separations}
\begin{thm}\label{ABc comp latt}
For every digraph $ D=(V,E) $ and $ A,B\subseteq V $, $ \mathfrak{S}_D(A,B)  $ is a nonempty complete lattice (with 
respect to the restriction of $  \trianglelefteq $).
\end{thm} 
\begin{proof}

The non-emptiness of the subposet $ \mathfrak{S}(A,B)  $  of 
$ \mathfrak{M}(A,B) $ is exactly the following theorem.
\begin{thm}[R. Aharoni and E. Berger, \cite{aharoni2009menger}]\label{inf Menger}
 For any  (possibly infinite) digraph $ D=(V,E) $ and $ A,B\subseteq V$,  $ \mathfrak{S}_D(A,B)\neq \varnothing  $.
 \end{thm}

\begin{prop}\label{larger AB cons}
If $ S\in \mathfrak{S}^{+}(A,B)  $, then $ \mathfrak{S}(A,S) =\{T\in \mathfrak{S}(A,B): T \trianglelefteq S  \} $.
\end{prop} 
\begin{sbiz}
Let $ T\in  \mathfrak{S}(A,S)  $. Take a $ \mathcal{P}\in \mathfrak{D}(A,S) $ with $ T\bot \mathcal{P} $. Since $ 
S\in \mathfrak{S}^{+}(A,B)  $, we can continue forward the paths $ \mathcal{P} $ to obtain an element of $ 
\mathfrak{D}(A,B) 
$ which shows $ T\in \mathfrak{S}(A,B)  $. Assume now that $ T\in \mathfrak{S}(A,B) $ with $ T \trianglelefteq  S  
$. Take a $ 
\mathcal{Q}\in \mathfrak{D}(A,B) $ with $ T\bot \mathcal{Q} $. The initial segments of the paths  $ \mathcal{Q} $ up to $ 
S 
$ show $ T\in \mathfrak{S}(A,S) $.
\end{sbiz}

\begin{lem}\label{closed sup inf}
$ \mathfrak{S}^{+}(A,B)  $ is closed under the $ \inf $ operation of $ \mathfrak{M}(A,B) $ 
and 
$ 
\mathfrak{S}^{-}(A,B)  $ is closed under $ \sup$.
\end{lem}
\begin{sbiz}
Let $\{ S_\xi \}_{\xi<\kappa}\subseteq \mathfrak{S}^{+}(A,B)  $ be nonempty and let $ S_{<\alpha}:=\inf \{ 
S_\xi: \xi<\alpha \} $. For every $ 0<\alpha\leq\kappa $ 
and 
every $ s\in S_{<\alpha} $  we  define a path $ P_s $ that goes from $ s $ to $ B $ such that for each $ \alpha $ the paths 
$ \{ 
P_s: s\in S_{<\alpha} \} $ are disjoint and hence 
witness $ S_{<\alpha}\in \mathfrak{S}^{+}(A,B)  $.

For $ \alpha=1 $ we pick an arbitrary path-system that witnesses $ S_0\in  \mathfrak{S}^{+}(A,B)  $. If $ \alpha $ is a limit 
ordinal and $ P_s $ is defined whenever $ s\in S_{<\xi} $ for some $ \xi<\alpha $, then from the characterisation of $ \inf  $
(see Proposition \ref{inf char}) it 
follows 
that $ P_s $ is defined for every $ s\in S_{<\alpha} $. If $ s\neq s'\in S_{<\alpha} $, then for every large enough $ \xi $ we 
have
$ s,s'\in S_{<\xi} $, thus by the induction hypothesis $ P_s $ and $ P_{s'} $ are disjoint. Suppose now that $ \alpha=\beta+1 
$.
Every $ s\in S_{<\beta+1}\setminus S_{<\beta} $ is in $ S_{\beta} $ hence we may fix a
$ \{ Q_s: s\in  S_{<\beta+1}\setminus S_{<\beta}\}\in \mathfrak{D}(S_\beta,B) $ where $ Q_s $ goes from $ s $ to $ B $. 
Since $ S_{<\beta} $ separates $ 
B $ from $ S_{<\beta+1} $ (see Proposition \ref{simm ord}), each $ Q_s $ meets $ S_{<\beta} $. Assume that the first 
common vertex of 
$ Q_s $ with $ 
S_{<\beta} $ is 
$ s' $. Note that $ s'\notin S_{<\beta+1} $ because $ S_\beta $ separates $ s' \notin S_{\beta}$ from $ A $. Unite  the initial 
segment of $ Q_s $ 
up to $ s' $ and $ P_{s'} $ to obtain $ P_s $.
\end{sbiz}

\begin{claim}\label{larges smallest}
$ \mathfrak{S}(A,B)  $ has a smallest and a largest  element, namely $ \inf\mathfrak{S}^{+}(A,B)  $ and $ 
\sup\mathfrak{S}^{-}(A,B)  $.
\end{claim}
\begin{sbiz}
Let $ S:=\inf \mathfrak{S}^{+}(A,B) $. By Lemma \ref{closed sup inf}, $ S\in \mathfrak{S}^{+}(A,B)  $.  
By Proposition \ref{larger AB cons}, 
$ \mathfrak{S}(A,S) =\{T\in \mathfrak{S}(A,B): T\trianglelefteq S  \} $. Since $ \mathfrak{S}(A,B) 
\subseteq\mathfrak{S}^{+}(A,B)  $, the set $ \{T\in \mathfrak{S}(A,B): T\trianglelefteq S  \}  $ cannot have an element 
strictly smaller 
than $ S $. By  Theorem \ref{inf Menger}, $ \mathfrak{S}(A,S)\neq \varnothing $, thus its only element must be $ S $.
\end{sbiz}\\

Let $ \mathcal{S}\subseteq \mathfrak{S}(A,B)  $ be nonempty. Since $ \mathfrak{S}(A,B) \subseteq 
\mathfrak{S}^{+}(A,B)  $ and by Lemma 
\ref{closed sup inf} $ \mathfrak{S}^{+}(A,B) $ is closed under the $ \inf $ operation of $ \mathfrak{M}(A,B) $,  
 $\inf \mathcal{S}=:S\in  \mathfrak{S}^{+}(A,B)$. Being smaller or equal to 
all  the elements of $ \mathcal{S} $ means being 
smaller 
or equal to $ S $. By Proposition \ref{larger AB cons}, the lower bounds of $ S $ in $ \mathfrak{S}(A,B)  $ are exactly the 
elements 
of $ \mathfrak{S}(A,S) $ which has a largest element by Claim \ref{larges smallest}.  It is the desired largest lower bound
of $ \mathcal{S} $ with respect to the poset $ \mathfrak{S}(A,B) $.

\begin{rem}
Theorem \ref{ABc comp latt}  remains true if the graph is undirected or if we consider  cuts instead of separations. The 
proof is essentially the 
same.
\end{rem}
\end{proof}
\subsection{Representation as Erd\H os-Menger separation lattice}
\begin{thm}\label{AB-cut rep}
Every  complete lattice is representable as an Erd\H os-Menger separation lattice.
\end{thm}

\begin{proof}
We reduce our theorem to the following theorem of Escalante.
\begin{thm}[Escalante, \cite{escalante1972schnittverbande}]\label{rep mincut}
For every complete lattice $ L $, there is a digraph $ D=(V,E) $ and $ A,B\subseteq V $ such that 
$ \mathfrak{M}_D(A,B) $ is isomorphic to $ L $.
\end{thm}
\begin{rem}
 For an English 
source, 
see Theorem 6 on page 157 of \cite{sauer1993finite}. It has been formulated originally for undirected graphs.
\end{rem}

Let $ L $ be a given complete lattice. First we choose $ D=(V,E),A,B $ according to Theorem \ref{rep mincut}. The only 
thing we need to do is to blow up the vertices of this system.  Indeed, consider 
$ V'=V\times \kappa $ where $ \kappa:=\left|V\right| +\aleph_0$ and 
draw an edge from $ (u,\alpha)$  to $(v,\beta) $ iff $ uv\in E $  to obtain $ D'=(V',E') $. We define $ A' $ to be $ A\times 
\kappa $ 
and 
$ B' $ to be $ B\times \kappa $. 

Note that if $(v,\alpha)\in S\in \mathfrak{M}_{D'}(A',B') $ then necessarily 
$ \{ v \}\times \kappa \subseteq S $ otherwise $ (v,\alpha) $  would be omittable in $ S $ contradicting its $ \subseteq 
$-minimality. It implies that $ T'\in  \mathfrak{M}_{D'}(A',B') $ iff there is a $ T\in 
\mathfrak{M}_{D}(A,B) $ 
such 
that $ T'=T\times \kappa $. Therefore $ \mathfrak{M}_{D'}(A',B')\cong \mathfrak{M}_{D}(A,B) $. It is enough to show 
that  $ \mathfrak{M}_{D'}(A',B')=\mathfrak{S}_{D'}(A',B')  $. To prove the non-trivial inclusion, take an arbitrary $ 
T'\in \mathfrak{M}_{D'}(A',B') $. Then $ T'=T\times \kappa $ for some  $T\in \mathfrak{M}_{D}(A,B) $.
For $ t\in T $, we can pick an $ A \rightarrow B $ path  $ P_t=v_0,\dots ,v_i, t, v_{i+1}, \dots, v_{n_t} $ in $ D $ where $ 
v_j\notin T $. Take an injection $ f: T \times \kappa \rightarrow \kappa $. Let $ P_{(t,\alpha)} $ that we obtain from $ P_t $ 
by replacing $ t $ with
$ (t,\alpha) $ and $ v_j $ by $ (v_j, f(t,\alpha)) $. It is easy to check that the path-system $ \{ P_{(t,\alpha)}: (t,\alpha)\in T' \} 
$ exemplifies $ T'\in \mathfrak{S}_{D'}(A',B') $.
\end{proof}

\subsection{Appendix}\label{appendix}

We show that $ \mathfrak{S}_D(A,B)  $ is not necessarily a sublattice of $ \mathfrak{M}_D(A,B) $.  Consider the 
digraph $ D $ and vertex sets $ A,B $ at Figure 
\ref{figur1}.
 We have $ S:=\{ \dots, b_{-2}, b_{-1},u, a_1, a_2, \dots \}\in \mathfrak{S}(A,B)  $ witnessed by
 \[ \{\dots a_{-2}b_{-2},a_{-1}b_{-1},a_0u b_1, a_1b_2, a_2 b_3, \dots   \}\in \mathfrak{D}(A,B).  \] We also have 
 $ T:=\{ \dots, a_{-2}, a_{-1},v, b_1, b_2, \dots \}\in \mathfrak{S}(A,B)  $ witnessed by
  \[ \{\dots a_{-2}b_{-3},a_{-1}b_{-2},a_0v b_{-1}, a_{1}b_{1}, a_{2} b_{2}, \dots   \}\in \mathfrak{D}(A,B).  \]
 Here 
$ \inf \{ S,T \}=(A\setminus \{ a_0 \})\cup \{ u,v \} $.
   
   But an $ A\rightarrow B $ path through $ u $ must start at $ a_0$ as well as an $ A\rightarrow B $ path through $ v $, 
   thus  $ (A\setminus \{ a_0 \})\cup \{ u,v \}\notin \mathfrak{S}(A,B) $.
\begin{figure}[H]
 \centering
 
 \begin{tikzpicture}

  \node at (-0.5,1.7) {$ A $};
 \node at (2.5,1.7) {$ B $};
 
\draw  (-1,1.5) rectangle (0,-4.5);

 \node (v1) at (-0.5,-1.5) {$ a_0 $};


 \node (v2) at (1,-1) {$ u $};

 \node (v3) at (1,-2) {$ v $};

 \draw  (2,1.5) rectangle (3,-4.5);

 \node (v6) at (-0.5,1) {$ \vdots $};

 \node (v8) at (-0.5,0.5) {$ a_3 $};

 \node (v12) at (-0.5,-0.5) {$ a_1 $};

 \node (v10) at (-0.5,0) {$ a_2 $};

 \node (v14) at (-0.5,-2.5) {$ a_{-1}$};

 \node (v16) at (-0.5,-3) {$ a_{-2} $};

 \node (v18) at (-0.5,-3.5) {$ a_{-3} $};

 \node at (-0.5,-3.8) {$ \vdots $};

 \node (v7) at (2.5,1) {$ \vdots $};

 \node (v11) at (2.5,0.5) {$ b_3 $};

 \node (v13) at (2.5,0) {$ b_2 $};

 \node (v4) at (2.5,-0.5) {$ b_{1} $};

 \node (v5) at (2.5,-2.5) {$ b_{-1} $};

 \node (v15) at (2.5,-3) {$ b_{-2}$};

 \node (v17) at (2.5,-3.5) {$ b_{-3}$};

 \node at (2.5,-3.8) {$ \vdots $};

 \draw  (v1) edge[->] (v2);

 \draw  (v1) edge[->] (v3);

 \draw  (v2) edge[->] (v4);

 \draw  (v3) edge[->] (v5);

 \draw  (v12) edge[->] (v13);

 \draw  (v12) edge[->] (v4);

 \draw  (v10) edge[->] (v13);

  \draw  (v10) edge[->] (v11);

 \draw  (v8) edge[->] (v11);

 \draw  (v14) edge[->] (v5);

 \draw  (v14) edge[->] (v15);

 \draw  (v16) edge[->] (v15);

 \draw  (v16) edge[->] (v17);

 \draw  (v18) edge[->] (v17);

 \end{tikzpicture}
 \caption{A system with $ S,T\in \mathfrak{S}(A,B)  $ where  
 $ \inf \{ T,S \}\notin \mathfrak{S}(A,B)  $ }\label{figur1}
 \end{figure}
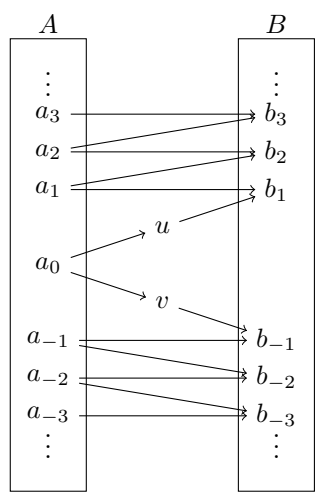

\end{document}